\documentclass[preprint,12pt]{elsarticle}

\usepackage{amssymb}
\usepackage{amsfonts}
\usepackage{amsthm}

\usepackage{natbib}

\usepackage{amsmath,amsthm,amsfonts,amssymb,color}
\usepackage[matrix,arrow]{xy}
\usepackage{graphicx}
\usepackage[UKenglish]{babel}
\usepackage{multicol}
\usepackage{appendix}
\usepackage{epstopdf}

\usepackage{caption}
\usepackage{subcaption}

\usepackage{multirow} % para las tablas
\usepackage{float}
\usepackage{graphicx}
\usepackage{longtable}

\newcommand\pder[2][]{\ensuremath{\frac{\partial#1}{\partial#2}}}

\newcommand{\Ste}{\text{Ste}}

\DeclareMathOperator\erf{erf}

 \newtheorem{thm}{Theorem}[section]
 
 \newtheorem{lem}[thm]{Lemma}
 
 \theoremstyle{definition}
 
 \theoremstyle{remark}
 \newtheorem{rem}[thm]{Remark}
 
 \numberwithin{equation}{section}

\begin{document}

\begin{frontmatter}



\title{Explicit solution for non-classical one-phase Stefan problem with variable thermal coefficients and two different heat source terms}

 \author[label1,label2]{Julieta Bollati \corref{cor1}}
 \author[label1]{Mar\'ia F. Natale}
 \author[label1]{Jos\'e A. Semitiel }
 \author[label1,label2]{Domingo A. Tarzia}
\address[label1]{Depto de Matem\'atica, FCE-Universidad Austral,  Paraguay 1950, 2000~Rosario, Argentina}
 \address[label2]{CONICET, Argentina}
\cortext[cor1]{jbollati@austral.edu.ar, (+54) 341 5223000}

\begin{abstract}
A one-phase Stefan problem for a semi-infinite material is investigated for special functional forms of the thermal conductivity and specific heat depending on the temperature of the phase-change material. Using the similarity transformation technique, an explicit solution for these situations are showed. The mathematical analysis is made for two different kinds of heat source terms, and the existence and uniqueness of the solutions are proved.
\end{abstract}

\end{frontmatter}

\section{Introduction.}

The phase change problems that contain one or more moving boundaries have attracted growing attention in the last decades due to their wide range of engineering,  industrial applications and natural sciences.  Stefan problems can be modelled as basic phase-change processes where the location of the interface is a priori an unknown function \cite{AlSo,Ca,CaJa,Cr,Gu,Lu,Ta2011}.

The present study provides the existence and uniqueness of solution of the similarity type to a one-phase Stefan fusion problem for a semi-infinite material where it is assumed a Dirichlet condition at the fixed face $x=0$ and it is governed by a non-classical and nonlinear heat equation with temperature-dependent thermal conductivity and specific heat coefficients, and two different kinds of heat source terms.

Non-classical heat conduction problem are considered when the source term is linear or nonlinear depending of the heat flux or the temperature on the boundary of the domain according to the corresponding boundary condition imposed. The non-classical problem are motivated by the modelling of a system of temperature regulation in isotropic media and the source term describes a cooling or a heating effect depending of different types of sources which are related to the evolution of the unknown boundary condition on the boundary of the domain. Problems of this type are related to the thermostat problem \cite{CaYi1989, FrJi1988, FuMiKe1996, GlSp1981, GlSp1982, Ke1990, KePr1988}. For example, we will use mathematical ideas developed for the one-dimensional case in \cite{BeTaVi2000, CeTaVi2015, SaTaVi2011, TaVi1998, Vi1986} and for the n-dimensional case in \cite{BoTa2014, BoTa2017, BoTa2020}. The first paper connecting the non-classical heat equation with a phase-change process (i.e. the Stefan problem) was \cite{BrTa2006a} and after this some other works on the subject were published, for example \cite{BoKh2021, BrNa2014, BrNa2019, BrTa2006b, BrTa2010a}. Moreover, in \cite{BrTa2010b} explicit solutions for the non-classical one-phase Stefan problem was given for cases corresponding to different boundary conditions on the fixed face $x=0$: temperature, heat flux and convective boundary condition.

The mathematical model of the governing phase-change process is described as follows:
\begin{align}
& \rho c(\theta) \pder[\theta]{t}=\frac{\partial}{\partial x} \left(k(\theta)\pder[\theta]{x} \right)-H, & 0<x<s(t), \quad t>0, \label{EcCalor}\\
&  \theta(0,t)=\theta_{_0}>\theta_f &t>0, \label{CondBorde}\\
&  \theta(s(t),t)=\theta_{f}, &t>0, \label{TempCambioFase}\\
&  k_0\pder[\theta]{x}(s(t),t)=-\rho l \dot s(t), &t>0, \label{CondStefan}\\
& s(0)=0,\label{FrontInicial}
\end{align}
where the unknown functions are the temperature $\theta=\theta(x,t)$ and the free boundary $x=s(t)$ separating both phases (the liquid phase at temperature $\theta(x,t)$ and the solid phase at constant temperature $\theta_{f}$ ). The parameters $\rho>0$ (density), $l>0$ (latent heat per unit mass), $\theta_{0}>0$ (temperature imposed at the fixed face $x=0$) and $\theta_f$ (phase change temperature at the free boundary $x=s(t)$) are all known constants.

If the thermal coefficients of the material are temperature-dependent, we have a doubly non-linear free boundary problem.  The functions $k$ and $c$ are defined as:
\begin{align}
&k(\theta)=k_{0}\left(1+\delta\left(\tfrac{\theta-\theta_{f}}{\theta_{0}-\theta_{f}}\right)^{p}\right)\label{k}\\
&c(\theta)=c_{0}\left(1+\delta\left(\tfrac{\theta-\theta_{f}}{\theta_{0}-\theta_{f}}\right)^{p}\right)\label{c},
\end{align}
where $\delta$ and $p$ are given non-negative constants, $k_{0}=k\left(\theta_f\right)$ and $c_{0}=c\left(\theta_f\right)$  are the reference thermal conductivity and the specific heat coefficients, respectively.

Some other models involving temperature-dependent thermal conductivity can also be found in \cite{BoNaSeTaLibro, BoNaSeTa-ThSci,BrNa3, BrNa, CST, CeSaTa2020,ChCh, ChSu74,  KSR, Ma, NaTa, NaTa2003, OlSu87, Ro85, Ro15, Ro18}.

Existence and uniqueness to the problem (\ref{EcCalor})-(\ref{FrontInicial}) with null source term, $H=0$, was developed in \cite{BoNaSeTa2020}.

The control function $H$ represents a heat source term for the nonlinear heat equation. Several applied papers give us the significance of the source term in the interior of the material which can undergo a change of phase \cite{BoTa2000, BrNaTa2007, Sc1994, TaVi1998, Vi1986}.  In this paper we considered two different control functions $H$. The first one is defined as in \cite{BrNaTa2007, BrTa2006b, MeTa1993} and the second one depends on the evolution of the heat flux at the fixed face $x=0$ like in \cite{BrNa2014, BrNa2019}. In this last case, we have a non-classical heat equation as in \cite{TaVi1998, Vi1986}.

We are interested in obtaining a similarity solution to problem (\ref{EcCalor})-(\ref{FrontInicial}) in which the temperature $\theta=\theta(x,t)$ can be written as a function of a single variable.  Through the following change of variables:
\begin{equation}
y(\eta)=\tfrac{\theta(x,t)-\theta_{f}}{\theta_{0}-\theta_{f}}\geq 0  \label{Y}
\end{equation}
with
\begin{equation}
\eta=\tfrac{x}{2a\sqrt{t}},\quad 0<x<s(t),\quad t>0, \label{eta}
\end{equation}
the phase front moves as
\begin{equation}
s(t)=2a\lambda\sqrt{t} \label{freeboundary}
\end{equation}
where $a^{2}=\frac{k_{0}}{\rho c_{0}}$ (thermal diffusivity) and $\lambda>0$ is a positive parameter to be determined.

The plan of this paper is the following. In Section 2, we prove the existence and uniqueness of solution to the problem (\ref{EcCalor})-(\ref{FrontInicial}) considering the control function given by \cite{Sc1994}:
\begin{equation}
H=H_1(x,t)=\frac{\rho l}{t}\beta\left(\frac{x}{2a\sqrt{t}}\right) \label{H1},
\end{equation}
where $\beta=\beta(\eta)$ in a function with appropriate regularity properties \cite{BoTa2000, BrNaTa2007, MeTa1993, Sc1994}.  Moreover, a particular case where $\beta$ is of exponential type given by
\begin{equation}
\beta(\eta)=\dfrac{1}{2}\exp(-\eta^2),
\end{equation}
is also studied in details. This type of heat source term is important through the use of microwave energy following \cite{Sc1994}.

Finally,  in Section 3, we prove existence and uniqueness of solution to the problem (\ref{EcCalor})-(\ref{FrontInicial}) considering the control function given by
\begin{equation}
H=H_2(t)=\frac{\lambda_0 }{\sqrt{t}}\frac{\partial T}{\partial x}(0,t)
\end{equation}
that can be thought of  by modelling of a system of temperature regulation in isotropic mediums \cite{BrNa2014, BrNa2019, BrTa2006b} with nonuniform source term, which provides a cooling or helting effect depending upon the properties of $H_2$ related to the heat flux (or the temperature in other cases) at the fixed face boundary $x=0$.

\section{Free boundary problem when the heat source term is of a similarity type}
We consider now the control function $H$ given by (\ref{H1}).

% $$H=H_1(x,t)=\frac{\rho l}{t}\beta\left(\frac{x}{2a\sqrt{t}}\right)$$

\subsection{General case}
We consider the hypothesis:
\\
$H_\beta$: $\beta=\beta(\eta)$ is an integrable function in $\left(0,\epsilon\right)$ for all $\epsilon >0$ and $\beta(\eta)\exp\left(\eta^2\right)$ is an integrable function in $\left(M,+\infty\right)$ for all $M>0$ \label{Hbeta}.

It is easy to see that the Stefan problem $\mathrm{(}$\ref{EcCalor}$\mathrm{)}$-(\ref{FrontInicial}) has a similarity solution $(\theta,s)$ given by:
\begin{align}
&\theta(x,t)=\left(\theta_{0}-\theta_{f}\right)y\left(\tfrac{x}{2a\sqrt{t}}\right)+\theta_{f},\quad  0<x<s(t), \quad t>0,\label{T1} \\
&s(t)=2a\lambda\sqrt{t},\quad\quad t>0\label{s1}
\end{align}
if and only if the function $y$ (defined by (\ref{Y})) and the parameter $\lambda>0$ (defined in (\ref{freeboundary})) satisfy the following ordinary differential problem:
\begin{align}
&2\eta(1+\delta y^p(\eta))y'(\eta)+[(1+\delta y^p(\eta))y'(\eta)]'=\tfrac{4}{\Ste}\beta(\eta), \quad &0<\eta<\lambda, \label{y1}\\
&y(0)=1,\label{cond01}\\
&y(\lambda)=0, \label{condlambda1}\\
&y'(\lambda)=-\tfrac{2\lambda}{\Ste} \label{eclambda1}
\end{align}
where $\delta > 0$, $p> 0$ and $\Ste=\tfrac{c_{0}(T_{0}-T_f)}{l}>0$ is the Stefan number.

\begin{lem}\label{ProbAux1}
Fixed $p > 0$, $\delta > 0$, $\lambda>0$, $y\in C^{\infty}[0,\lambda]$, $y\geq 0$, and $\beta=\beta(\eta)$  verifies the hypothesis $H_\beta$.

%an integrable function in $\left(0,\epsilon\right)$ for all $\epsilon >0$ and $\beta(\eta)\exp\left(\eta^2\right)$ %an integrable function in $\left(M,+\infty\right)$ for all $M>0$.

Then, $(y,\lambda)$ is a solution to the ordinary differential equation (\ref{y1})-(\ref{eclambda1}) if and only if  $\lambda>0$ is a solution to the equation:
\begin{align}\label{7-1}
&\tfrac{\sqrt{\pi}}{{\Ste}} x \erf(x) \exp(x^2) +\tfrac{2\sqrt{\pi}}{\Ste} \int_{0}^{x} \exp(\xi^2)\erf(\xi)\beta(\xi)\, d\xi =1+\tfrac{\delta}{p+1}, \quad x>0,
\end{align}
and the function $y=y(\eta)$ satisfies the functional equation:
\begin{align}\label{8-1}
y(\eta)\left(1+\tfrac{\delta}{p+1}y^p(\eta)\right)&=1+\tfrac{\delta}{p+1}-\tfrac{\sqrt{\pi} \erf(\eta)}{\mathrm{Ste}}\left(2\int_{0}^{\lambda} \beta(\xi) \exp(\xi^2)\, d\xi - \lambda \exp(\lambda^2)\right)+\nonumber\\
&+\tfrac{2\sqrt{\pi}}{\Ste}\left( \int_{0}^{\eta} \beta(\xi) \exp(\xi^2) \left( \erf(\eta) - \erf(\xi)\right)\, d\xi\right), \ 0 \leq \eta \leq \lambda
\end{align}
 \end{lem}

\begin{proof}
Let $(y,\lambda)$ be a solution to (\ref{y1})-(\ref{eclambda1}). We define the function:
\begin{equation}
v(\eta)=\left(1+\delta y^{p}(\eta) \right) y'(\eta). \label{v-eta-1}
\end{equation}

Taking into account  (\ref{y1}) and the condition  (\ref{cond01})  the function $v$ can be rewrite as
\begin{equation}
v(\eta)=\exp{(-\eta^2)}\left(\tfrac{4}{\Ste}\int_{0}^{\eta} \beta(\xi)\exp(\xi^2)\, d\xi+ (1+\delta)y'(0)\right). \label{v-eta1-2}
\end{equation}

From (\ref{v-eta-1}) and (\ref{v-eta1-2}), we obtain
\begin{equation} \label{ec2-11-1}
\left(1+\delta y^{p}(\eta) \right) y'(\eta)=\exp{(-\eta^2)}\left(\tfrac{4}{\Ste}\int_{0}^{\eta} \beta(\xi)\exp(\xi^2)\, d\xi+ (1+\delta)y'(0)\right).
\end{equation}

Taking $\eta=\lambda$ in the above equation, using (\ref{condlambda1}) and (\ref{eclambda1}), we obtain:
\begin{equation}
y'(0)=-\tfrac{4}{\Ste(1+\delta)}\left(2\int_{0}^{\lambda} \beta(\xi)\exp(\xi^2)\, d\xi+\lambda \exp{(\lambda^2)}\right) \label{yprima01}
\end{equation}

Integrating equation (\ref{ec2-11-1}) in the domain $(0,\eta)$  and by virtue of (\ref{cond01}), we obtain:

\begin{align}
y(\eta)\left(1+\tfrac{\delta}{p+1}y^p(\eta)\right)&=1+\tfrac{\delta}{p+1}+(1+\delta)y'(0)\tfrac{\sqrt{\pi}}{2}\erf(\eta)+\nonumber\\
&+\tfrac{4}{\Ste}\int_{0}^{\eta} \int_{\xi}^{\eta} \beta(\xi)\exp(-z^2)\exp(\xi^2)\, dz\ d\xi
\end{align}
Given that $\int_{0}^{\eta} \int_{\xi}^{\eta} \beta(\xi)\exp(-z^2)\exp(\xi^2)\, dz\ d\xi =\tfrac{\sqrt{\pi}}{2} \int_{0}^{\eta} \left( \erf(\eta)-\erf(\xi)\right)\beta(\xi)\exp(\xi^2)\, d\xi$ and from (\ref{yprima01}), we obtain that $y=y(\eta)$ is a solution to  (\ref{8-1}).

\medskip

Taking $\eta=\lambda$ in equation (\ref{8-1}) and using (\ref{condlambda1}), we conclude that $\lambda>0$ is a solution to equation (\ref{7-1}).

\medskip

Reciprocally, if $(y,\lambda)$ is a solution to  (\ref{7-1})-(\ref{8-1}), then
\begin{align*}
y(\eta)&=1+\tfrac{\delta}{p+1}-\tfrac{\delta}{p+1}y^{p+1}(\eta)-\tfrac{\sqrt{\pi} \erf(\eta)}{\mathrm{Ste}}\left(2\int_{0}^{\lambda} \beta(\xi) \exp(\xi^2)\, d\xi - \lambda \exp(\lambda^2)\right)+\nonumber\\
&+\tfrac{2\sqrt{\pi}}{\Ste}\left( \int_{0}^{\eta} \beta(\xi) \exp(\xi^2) \left( \erf(\eta) - \erf(\xi)\right)\, d\xi\right),
\end{align*}

and it follows immediately that  $(y,\lambda)$ is a solution to (\ref{y1})-(\ref{eclambda1}).

\end{proof}

\begin{lem}\label{ProbAux11}
If $p > 0$, $\delta > 0$, $\beta=\beta(\eta)\geq 0$  verifies the hypothesis $H_\beta$, then there exists a unique solution $(y,\lambda)$ to the functional problem defined by (\ref{7-1})-(\ref{8-1}) with $\lambda>0$, $y\in C^{\infty}[0,\lambda]$ and $y\geq 0$.
\end{lem}
%an integrable function in $\left(0,\epsilon\right)$ for all $\epsilon >0$ and $\beta(\eta)\exp\left(\eta^2\right)$ %an integrable function in $\left(M,+\infty\right)$ for all $M>0$,

\begin{proof}
For $x>0$, we define the function given by $$\varphi_1(x)=\tfrac{\sqrt{\pi}}{{\Ste}} x \erf(x) \exp(x^2) +\tfrac{2\sqrt{\pi}}{\Ste} \int_{0}^{x} \exp(\xi^2)\erf(\xi)\beta(\xi)\, d\xi .$$

It is easy to see that $\varphi_1(0)=0$, $\varphi_1(+\infty)=+\infty$ and $\varphi_1$ is an increasing function. Then, there exists a unique  $\lambda>0$ solution to equation  (\ref{7-1}): $$\varphi_1(x)= 1+\tfrac{\delta}{p+1}.$$

Let $\Phi$ be the function defined by 
\begin{equation}\label{funcionphi}
 \Phi(x)=x+\tfrac{\delta}{p+1}x^{p+1}, \quad 0\leq x \leq1.
\end{equation}
  Clearly,  $\Phi(0)=0$, $\Phi(1)=1+\tfrac{\delta}{p+1}$ and $\Phi$ in an increasing function then, there exists the inverse function $\Phi^{-1}:[0,1+\tfrac{\delta}{p+1}] \to [0,1]$.
\medskip

For the unique value $\lambda>0$ obtained in (\ref{7-1}), we can define the function
\begin{align*}
\Psi_1(x)&=1+\tfrac{\delta}{p+1}-\tfrac{\sqrt{\pi} \erf(x)}{\mathrm{Ste}}\left(2\int_{0}^{\lambda} \beta(\xi) \exp(\xi^2)\, d\xi - \lambda \exp(\lambda^2)\right)+\nonumber\\
&+\tfrac{2\sqrt{\pi}}{\Ste}\left( \int_{0}^{x} \beta(\xi) \exp(\xi^2) \left( \erf(x) - \erf(\xi)\right)\, d\xi\right), \quad x \in \left[0,\lambda\right].
\end{align*}
Then $\Psi_1(0)=1+\tfrac{\delta}{p+1}$,  $\Psi_1(\lambda)=0$ and $\Psi_1$ is a decreasing function. Furthermore, $\Psi_1(x) \in \left[0,1+\tfrac{\delta}{p+1}\right]$ for all $x \in \left[0,\lambda\right]$.

\medskip

We conclude that there exists a unique function $y\in C^{\infty}[0,\lambda]$ solution to the equation  $$\Phi\left(y(\eta)\right)=\Psi_1(\eta),$$ given by
\begin{equation}
y(\eta)=\Phi^{-1}\left(\Psi_1(\eta)\right), \quad 0\leq \eta \leq \lambda.
\end{equation}
\end{proof}

From the above lemmas we are able to claim the following result:

\begin{thm} \label{ExistenciaDirichlet-1}
The Stefan problem governed by $($\ref{EcCalor}$)$-$($\ref{FrontInicial}$)$ has a unique similarity type solution given by $($\ref{T1}$)$-$($\ref{s1}$)$ where $(y,\lambda)$ is the unique solution to the functional problem $($\ref{7-1}$)$-$($\ref{8-1}$)$.
\end{thm}

\begin{rem} \label{2.4-1}
On one hand we have that $\Phi$ is an increasing function with $\Phi(0)=0$ and $\Phi(1)=1+\frac{\delta}{p+1}$.  On the other hand, $\Psi_1$ is a decreasing function with  $\Psi_1(0)=1+\frac{\delta}{p+1}$ and $\Psi_1(\lambda)=0$. Then it follows that $0\leq y(\eta)\leq 1 $, for $0\leq\eta\leq\lambda$.

In virtue of this and Theorem \ref{ExistenciaDirichlet-1} we have that
\begin{equation*}
\theta_f<\theta(x,t)<\theta_0,\qquad \qquad 0<x<s(t),\quad t>0.
\end{equation*}
\end{rem}

\subsection{Particular case}
If we consider the particular case where $\beta$ is of exponential type given by
$$\beta(\eta)=\dfrac{1}{2}\exp(-\eta^2)$$

\begin{thm} \label{ExistenciaDirichlet-11}
If $p > 0$, $\delta > 0$, the Stefan problem governed by $($\ref{EcCalor}$)$-$($\ref{FrontInicial}$)$ has a unique similarity type solution $(y,\lambda)$
with $\lambda>0$, $y\in C^{\infty}[0,\lambda]$ and $y\geq 0$, where  $\lambda>0$ is the unique solution to the equation:
\begin{align}
\tfrac{1}{\Ste}\left(1-\exp(-x^2)\right)+\tfrac{\sqrt{\pi}}{\Ste}x  ~\erf(x)\left(\exp(x^2)-1\right)=1+\tfrac{\delta}{p+1},\qquad \qquad x>0, \label{7-1-1}
\end{align}
and the function $y=y(\eta)$ satisfies the equation
\begin{align} \label{8-1-1}
y(\eta)\left(1+\tfrac{\delta}{p+1}y^p(\eta)\right)=1+\tfrac{\delta}{p+1}+\tfrac{1}{\Ste}\left( \exp(-\eta^2)-1\right)+   \tfrac{\sqrt{\pi}}{\mathrm{Ste}} \;\lambda ~ \erf(\eta)\left(1-\exp(\lambda^2)\right),\nonumber\\
0 \leq \eta \leq \lambda.
\end{align}

\end{thm}

\begin{rem}
For the particular case $p=1$, the unique function $y=y(\eta)$ solution to equation (\ref{8-1-1}) is given by
\begin{equation*}
y(\eta)=\tfrac{1}{\delta} \left[\sqrt{1+4\left(1+\tfrac{\delta}{2}\right)+\tfrac{1}{\Ste}\left(\exp(-\eta^2-1\right)+\tfrac{\sqrt{\pi}}{\Ste} \lambda \erf(\eta) \left(1-\exp(\lambda^2)\right)}-1 \right],
\end{equation*}
where $\lambda>0$ is the unique solution to equation (\ref{7-1-1}) for $p=1$.
\end{rem}

\section{Free boundary problem with a heat source that depends on the evolution of the heat flux at the fixed face $x=0$}
If we consider the control function $H$ depends on the evolution of the heat flux at the fixed face $x=0$, that is $$H=H_2(W(t),t)=\frac{\lambda_0 }{\sqrt{t}}\frac{\partial T}{\partial x}(0,t),$$
where $\lambda_0>0$, it is easy to see that the Stefan problem $\mathrm{(}$\ref{EcCalor}$\mathrm{)}$-(\ref{FrontInicial}) has a similarity solution $(\theta,s)$ given by:
\begin{align}
&\theta(x,t)=\left(\theta_{0}-\theta_{f}\right)y\left(\tfrac{x}{2a\sqrt{t}}\right)+\theta_{f},\quad  0<x<s(t), \quad t>0,\label{T} \\
&s(t)=2a\lambda\sqrt{t},\quad\quad t>0\label{s}
\end{align}
if and only if the function $y$ and the parameter $\lambda>0$ satisfy the following ordinary differential problem:
\begin{align}
&2\eta(1+\delta y^p(\eta))y'(\eta)+[(1+\delta y^p(\eta))y'(\eta)]'=Ay'(0), \quad &0<\eta<\lambda, \label{y}\\
&y(0)=1,\label{cond0}\\
&y(\lambda)=0, \label{condlambda}\\
&y'(\lambda)=-\tfrac{2\lambda}{\Ste} \label{eclambda}
\end{align}
where $\delta > 0$, $p> 0$, $A=\frac{2\lambda_0}{\rho c_0 a}$ and $\Ste=\tfrac{c_{0}(T_{0}-T_f)}{l}>0$ is the Stefan number.

\begin{lem}\label{ProbAux}
Fixed $p > 0$, $\delta > 0$, $\lambda>0$, $y\in C^{\infty}[0,\lambda]$ and $y\geq 0$. Then, $(y,\lambda)$ is a solution to the ordinary differential equation (\ref{y})-(\ref{eclambda}) if and only if  $\lambda$ is a solution to the equation:
\begin{align}\label{7}
&\tfrac{\sqrt{\pi}x \exp{(x^2)}}{\Ste \left(A \int_{0}^{x}  \exp(z^2)\, dz +1 +\delta \right) } \left(A\int_{0}^{x} \exp(z^2)\left(\erf(x)-\erf(z)\right)\, dz+(1+\delta)\erf(x)\right)=\nonumber\\
& \qquad \qquad \qquad \qquad \qquad \qquad \qquad \qquad \qquad \qquad \qquad \qquad =1+\tfrac{\delta}{p+1}, \quad x>0,
\end{align}
and the function $y=y(\eta)$ satisfies the equation:
\begin{align}\label{8}
y(\eta)\left(1+\tfrac{\delta}{p+1}y^p(\eta)\right)&=1+\tfrac{\delta}{p+1}- \tfrac{\sqrt{\pi} \lambda\exp(\lambda^2)}{\Ste\left(A\int_{0}^{\lambda} \exp(z^2)\, dz+1+\delta\right)}    \left(A\int_{0}^{\eta} \exp(z^2)\left(\erf(\eta)+\right.\right.\nonumber\\
&\left.\left.-\erf(z)\right)\, dz + (1+\delta)\erf(\eta)\right), \qquad \qquad \qquad 0\leq \eta\leq \lambda.
\end{align}
 \end{lem}

\begin{proof}
Let $(y,\lambda)$ be a solution to (\ref{y})-(\ref{eclambda}). We define the function:
\begin{equation}
v(\eta)=\left(1+\delta y^{p}(\eta) \right) y'(\eta). \label{v-eta}
\end{equation}

Taking into account  (\ref{y}) and the condition  (\ref{cond0})  the function $v$ can rewrite as
\begin{equation}
v(\eta)=\exp{(-\eta^2)}y'(0)\left(A\int_{0}^{\eta} \exp(z^2)\, dz+ 1+\delta\right). \label{v-eta1}
\end{equation}

From (\ref{v-eta}) and (\ref{v-eta1}), we obtain
\begin{equation} \label{ec2-11}
\left(1+\delta y^{p}(\eta) \right) y'(\eta)=\exp{(-\eta^2)}y'(0)\left(A\int_{0}^{\eta} \exp(z^2)\, dz+ 1+\delta\right).
\end{equation}

Taking $\eta=\lambda$ in the above equation, using (\ref{condlambda}) and (\ref{eclambda}), we obtain:
\begin{equation}
y'(0)=-\frac{2\lambda\exp(\lambda^2)}{\Ste\left(A\int_{0}^{\lambda} \exp(z^2)\, dz+1+\delta\right)} \label{yprima0}
\end{equation}

Integrating into $(0,\eta)$ equation (\ref{ec2-11}) and by virtue of (\ref{cond0}), we obtain:

\begin{align}
y(\eta)\left(1+\tfrac{\delta}{p+1}y^p(\eta)\right)&=1+\tfrac{\delta}{p+1}+(1+\delta)y'(0)\tfrac{\sqrt{\pi}}{2}\erf(\eta)+\nonumber\\
&+Ay'(0)\int_{0}^{\eta} \int_{z}^{\eta} \exp(-z^2)\exp(\xi^2)\, dz\ d\xi
\end{align}
Given that $\int_{0}^{\eta} \int_{z}^{\eta}\exp(-z^2)\exp(\xi^2)\, dz\ d\xi =\tfrac{\sqrt{\pi}}{2} \int_{0}^{\eta} \left( \erf(\eta)-\erf(z)\right)\exp(z^2)\, dz$ and from (\ref{yprima0}), we obtain that $y=y(\eta)$ is a solution to  (\ref{8}).

\medskip

Taking $\eta=\lambda$ in equation (\ref{8}) and using (\ref{condlambda}), we conclude that $\lambda>0$ is a solution to equation (\ref{7}).

\medskip

Reciprocally, if $(y,\lambda)$ is a solution to  (\ref{7})-(\ref{8}),
\begin{align*}
y(\eta)&=1+\tfrac{\delta}{p+1}-\tfrac{\delta}{p+1}y^{p+1}(\eta)- \tfrac{\sqrt{\pi} \lambda\exp(\lambda^2)}{\Ste\left(A\int_{0}^{\lambda} \exp(z^2)\, dz+1+\delta\right)}    \left(A\int_{0}^{\eta} \exp(z^2)\left(\erf(\eta)+\right.\right.\nonumber\\
&\left.\left.-\erf(z)\right)\, dz + (1+\delta)\erf(\eta)\right),
\end{align*}

and it follows immediately that  $(y,\lambda)$ is a solution to (\ref{y})-(\ref{eclambda}).

\end{proof}

\begin{lem}\label{ProbAux}
If $p > 0$ y $\delta > 0$, then there exists a unique solution $(y,\lambda)$ to the functional problem defined by (\ref{7})-(\ref{8}) with $\lambda>0$, $y\in C^{\infty}[0,\lambda]$ and $y\geq 0$.
\end{lem}
\begin{proof}
For $x>0$, we define the function given by $$\varphi_2(x)=\tfrac{\sqrt{\pi}x \exp{(x^2)}}{\Ste \left(A \int_{0}^{x}  \exp(z^2)\, dz +1 +\delta \right) } \left(A\int_{0}^{x} \exp(z^2)\left(\erf(x)-\erf(z)\right)\, dz+(1+\delta)\erf(x)\right).$$

It is easy to see that $\varphi_2(0)=0$, $\varphi_2(+\infty)=+\infty$ and $\varphi_2$ is an increasing function. Then, there exists a unique  $\lambda>0$ solution to equation  (\ref{7}): $$\varphi_2(x)= 1+\tfrac{\delta}{p+1}.$$

Let $\Phi$ be the function given by \eqref{funcionphi}.  Clearly,  $\Phi(0)=0$,  $\Phi(1)=1+\tfrac{\delta}{p+1}$ and $\Phi$ in an increasing function then, there exists the function $\Phi^{-1}:[0,1+\tfrac{\delta}{p+1}] \to [0,1]$.
\medskip

For the unique value $\lambda>0$ obtained in (\ref{7}), let the function
\begin{align*}
\Psi_2(x)&=1+\tfrac{\delta}{p+1}- \tfrac{\sqrt{\pi} \lambda\exp(\lambda^2)}{\Ste\left(A\int_{0}^{\lambda} \exp(z^2)\, dz+1+\delta\right)}\left(A\int_{0}^{x} \exp(z^2)\left(\erf(x)-\erf(z)\right)\, dz+\right.\\
&\left.+(1+\delta)\erf(x)\right), \quad x \in \left[0,\lambda\right].
\end{align*}
Then $\Psi_2(0)=1+\tfrac{\delta}{p+1}$,  $\Psi_2(\lambda)=0$ and $\Psi_2$ is a decreasing function.  Furthermore,  $\Psi_2(x) \in \left[0,1+\tfrac{\delta}{p+1}\right]$ for all $x \in \left[0,\lambda\right]$.

\medskip

We conclude that there exists a unique function $y\in C^{\infty}[0,\lambda]$ solution to the equation  $$\Phi\left(y(\eta)\right)=\Psi_2(\eta),$$ given by
\begin{equation}
y(\eta)=\Phi^{-1}\left(\Psi_2(\eta)\right), \quad 0 \leq \eta \leq \lambda.
\end{equation}
\end{proof}

From the above lemmas we are able to claim the following result:

\begin{thm} \label{ExistenciaDirichlet}
The Stefan problem governed by $($\ref{EcCalor}$)$-$($\ref{FrontInicial}$)$ has a unique similarity type solution given by $($\ref{T}$)$-$($\ref{s}$)$ where $(y,\lambda)$ is the unique solution to the functional problem $($\ref{7}$)$-$($\ref{8}$)$.
\end{thm}

\begin{rem} \label{2.4}
On one hand we have that $\Phi$ is an increasing function with $\Phi(0)=0$ and $\Phi(1)=1+\frac{\delta}{p+1}$.  On the other hand, $\Psi_2$ is a decreasing function with  $\Psi_2(0)=1+\frac{\delta}{p+1}$ and $\Psi_2(\lambda)=0$. Then it follows that $0\leq y(\eta)\leq 1 $, for $0 \leq \eta \leq \lambda$.

From this and Theorem \ref{ExistenciaDirichlet} we have that
\begin{equation*}
\theta_f<\theta(x,t)<\theta_0,\qquad \qquad 0<x<s(t),\quad t>0.
\end{equation*}
\end{rem}

\section{Conclusions}
One dimensional non-classical Stefan problems with temperature dependent thermal coefficients and a Dirichlet type condition at fixed face $x=0$ for a semi-infinite phase-change material were considered.  Existence and uniqueness of solution were obtained by using the similarity method and explicit solutions were found.
% ------------------------------------------------------------------------

\section*{Acknowledgement}
The present work has been partially sponsored by the Project PIP No 0275 from CONICET - Universidad Austral, Rosario, Argentina, and by the European Union’s Horizon 2020 Research and Innovation Programme under the Marie Skłodowska-Curie Grant Agreement No. 823731 CONMECH.

\end{document}